\documentclass[a4paper]{amsart}
\usepackage[T1]{fontenc}
\usepackage[utf8]{inputenc}

\usepackage{amsmath}
\usepackage{amssymb}
\usepackage{amsthm}

\usepackage{mathpazo,courier}
\usepackage[scaled]{helvet}

\usepackage[numbers,sort&compress]{natbib}

\usepackage{hyperref}
\usepackage{enumerate}

\newcommand{\N}{\mathbb{N}}
\newcommand{\Z}{\mathbb{Z}}

\newcommand{\R}{\mathbb{R}}
\newcommand{\C}{\mathbb{C}}

\newcommand{\A}{\mathbb{A}}
\renewcommand{\P}{\mathbb{P}}

\newcommand{\cV}{\mathcal{V}}
\newcommand{\cI}{\mathcal{I}}
\newcommand{\cC}{\mathcal{C}}

\newcommand\ph\varphi
\newcommand\ps\psi
\newcommand\ep\varepsilon
\newcommand\rh\varrho
\newcommand\al\alpha
\newcommand\be\beta
\newcommand\ga\gamma
\newcommand\om\omega
\newcommand\ta\tau
\renewcommand\th\vartheta
\newcommand\de\delta
\newcommand\ze\zeta
\newcommand\ch\chi
\newcommand\et\eta
\newcommand\io\iota
\newcommand\la\lambda
\newcommand\si\sigma

\newcommand\Ga\Gamma
\newcommand\De\Delta
\newcommand\Th\Theta
\newcommand\La\Lambda
\newcommand\Si\Sigma
\newcommand\Ph\Phi
\newcommand\Ps\Psi
\newcommand\Om\Omega

\DeclareMathOperator\im{im}

\DeclareMathOperator\Cl{Cl}
\DeclareMathOperator\Tr{Tr}
\DeclareMathOperator\Pic{Pic}
\DeclareMathOperator\Div{Div}

\DeclareMathOperator\Sym{Sym}

\newcounter{thmctr}
\newtheorem{thm}[thmctr]{Theorem}
\newtheorem{prop}[thmctr]{Proposition}

\newtheorem{lem}[thmctr]{Lemma}

\theoremstyle{definition}
\newtheorem{dfn}[thmctr]{Definition}
\newtheorem{rem}[thmctr]{Remark}
\newtheorem{ass}[thmctr]{Assumption}

\renewenvironment{proof}[1][\unskip]{\par\noindent {\em Proof #1: }}{{\qed\bigskip}}

\title[Definite determinantal Representations]{\normalsize Definite determinantal Representations of ternary hyperbolic Forms}
\date{\small\today}
\author{\small Christoph Hanselka}
\address{Universität Konstanz, 78457 Konstanz, Germany}
\email{christoph.hanselka@uni-konstanz.de}

\begin{document}
\begin{abstract}
	Hyperbolic polynomials are homogeneous polynomials whose zeros restricted to lines in certain directions are all real. These directions form a convex cone, the hyperbolicity cone. Examples of these hyperbolicity cones are spectrahedral cones, the feasible sets of semidefinite programming. Whether these are the only examples is unknown in general. It is true in the three-dimensional case as follows from the positive solution of the Lax conjecture, namely the existence of definite determinantal representations of ternary hyperbolic polynomials, proven by Helton and Vinnikov in 2007. We give the first completely algebraic proof of this fact, with mostly elementary methods.
\end{abstract}
\maketitle

\section{Introduction}

Coming from the theory of hyperbolic differential equations Peter Lax conjectured in \citep{Lax58} what was later essentially proven in \citep{Helton_Vinnikov07}:


\begin{thm}[Helton-Vinnikov]\label{thm:heltonvinnikov}
	Every ternary form that is hyperbolic with respect to $e\in\R^3$ admits a linear determinantal representation that is definite at $e$.
\end{thm}
\noindent
A form (i.e. a homogeneous polynomial) $F\in\R[X,Y,Z]$ of degree $d$ is called \emph{hyperbolic} with respect to some direction $e\in \R^3$ if $F(e)>0$ and for all $a\in\R^3$ the univariate polynomial $F(Te-a)\in\R[T]$ has only real roots. A \emph{linear determinantal representation} of $F$ is a matrix of the form $XA+YB+ZC$, where $A,B,C\in\R^{d\times d}$ such that
\[F=\det(XA+YB+ZC)\]
The representation is called \emph{definite} at $e=(x,y,z)$ if $xA+yB+zC$ is definite, in particular symmetric.

\medskip If $F$ is hyperbolic with respect to direction $e$, then the \emph{hyperbolicity cone} of $F$ is the connected component of $e$ in $\{\, a\in\R^3 \mid F(a)\neq0\,\}$. If $F$ admits a linear determinantal representation $M:=XA+YB+ZC$ definite at $e$, then this cone coincides with the \emph{spectrahedral cone}
\[\{\, a\in\R^3 \mid M(a)\text{ is positive definite}\,\}\]
of $M$. Semidefinite programming is optimization over such a cone. The problem of finding a spectrahedral description of hyperbolicity cones is strongly connected to finding definite determinantal representations of hyperbolic polynomials. For a survey about this and related topics see \citep{Vinnikov12}. The original proof of \citep[Theorem 2.2]{Helton_Vinnikov07} by Helton and Vinnikov is mostly based on results and methods from \citep{Vinnikov93} as well as \citep{Ball_Vinnikov99} using the theory of theta functions on the Jacobian variety of the compact Riemann surface defined by $F$. The formulation, as it appears in \citep{Helton_Vinnikov07}, differs from the way it is stated in Theorem \ref{thm:heltonvinnikov}. The equivalence of both versions and the fact that it solved the Lax conjecture has been noted in \citep{Lewis05}.

\medskip
The setup in our approach is based on an affine (nonhomogeneous) analogue to hyperbolicity and to the above notion of linear determinantal representations. In some sense it will be slightly more general, since it will also cover certain nonlinear determinantal representations. Our notions are similar to those in \citep{Vinnikov93} and other related articles, but have one technical advantage, which makes the involved arguments considerably easier.

A difference is the role that degrees play in the two kinds of determinantal representations. In \citep{Helton_Vinnikov07} and \citep{Vinnikov93} and others they are \textbf{linear} by construction. The strategy of the proof based on the work in \citep{Vinnikov93} is to classify all symmetric linear determinantal representations by relating them to certain real theta characteristics on the \textbf{projective} curve defined by the polynomial in question. One difficulty is then to show that among these representations there are also \textbf{definite} ones. The condition for such a theta characteristic to give rise to a determinantal representation is that it has no \textbf{global sections}. The arguments that are used in \citep{Vinnikov93} to show their existence involve the theory of theta functions on the Jacobian variety of the curve, which seems rather advanced and is nonalgebraic in nature. The objects in the present note that correspond to the above mentioned theta characteristics are certain fractional ideals over the coordinate ring of the \textbf{affine} curve defined by the polynomial. Their existence is shown in the proof of Theorem~\ref{thm:smoothrealrooted} and is based on Theorem~\ref{thm:divisible}, the divisibility of the class group of the curve, which is also not quite elementary, however much more well-known to algebraic geometers and can be treated in a completely algebraic manner. In particular, the proof works in the context of any real closed field instead of the field of real numbers $\R$. The problem of \textbf{global behavior} of the elements in these fractional ideals does not arise unless one wants to relate it directly to the \textbf{degree} of the entries in a determinantal representation of the polynomial. The latter can however be analyzed also with much more elementary methods. This will be done in Lemma~\ref{lem:proper}. The kind of determinantal representation that we are going to consider are \textbf{definite} by construction. And Lemma~\ref{lem:proper} leads to the observation that these are automatically \textbf{linear} in the case of hyperbolic polynomials.

The methods used in the following are partially inspired by work done by Bender in \citep{Bender67_rings} about characteristic polynomials of symmetric matrices over integral domains, and particularly over the ring of integers.

\begin{dfn}
	We call a univariate polynomial $f\in\R[T]$ \emph{(strictly) real rooted}, if it is monic and has only real (and simple) roots. A two variate polynomial $f\in\R[X,T]$ is \emph{(strictly) $T$-real rooted on $M\subseteq \R$}, if for all $x\in M$ the univariate polynomial $f(x,T)$ is (strictly) real rooted. If $M=\R$ we omit ``on $M$''.
\end{dfn}

\medskip
We do immediately get examples of such $T$-real rooted polynomials, namely the characteristic polynomials $\det(TI_d-A)$ of symmetric matrices $A\in\R[X]^{d\times d}$, as follows immediately by the well known fact that real symmetric matrices have only real eigenvalues. This leads to the following
\begin{dfn}
	Let $f\in\R[X,T]$ be monic and of degree $d$ in $T$. A \emph{$T$-spectral determinantal representation of $f$} is a matrix of the form $TI_d-A$, where $A\in \R[X]^{d\times d}$ such that
	\[f=\det(TI_d-A)\]
	It is called \emph{symmetric}, if $A$ is symmetric.
\end{dfn}
Note that the entries of $A$ can have arbitrary degree, i.e. they need not be linear. It is easy to see, that polynomials admitting a definite linear determinantal representation are hyperbolic. The Helton-Vinnikov Theorem states the converse. Likewise the main result of the present work is the following question. If we are given a $T$-real rooted polynomial in $\R[X,T]$, does it admit a symmetric $T$-spectral determinantal representation? We will give a positive answer to this problem, from which also Theorem~\ref{thm:heltonvinnikov} will follow.

\medskip
In Section~\ref{sec:basics} we will introduce the notation that is used throughout the text as well as some basic concepts and results from linear and bilinear algebra. Section~\ref{sec:smoothrealrooted} will contain the crucial step, concerning real rooted polynomials defining a smooth affine curve. Section~\ref{sec:approx} will show how one can approximate real rooted polynomials by strictly real rooted ones, defining a smooth curve. The result of the previous two sections will be combined in Section~\ref{sec:refine} yielding the main result, Theorem~\ref{thm:realrooted}, for arbitrary real rooted polynomials, and as a corollary Theorem~\ref{thm:heltonvinnikov}.

\section{Notation and some Linear Algebra}\label{sec:basics}
\begin{itemize}
	\item If not stated otherwise, $k$ and $d$ will always denote arbitrary natural numbers, i.e. elements of $\N=\{0,1,2,\dots\}$.
	\item Whenever a sum consists of quadratic matrices and scalars, the latter are to be understood as multiples of the identity matrix of the appropriate size. For example if $A\in S_d$ and $T$ is an indeterminate, we write $\det(T-A)$ for $\det(TI_d-A)$.
\end{itemize}
Given a commutative ring $R$, (in what follows, often $R=\R[X]$)
\begin{itemize}
	\item we write $f'$ for the derivative of $f\in R[T]$ with respect to $T$, i.e. $f'=\frac{\partial f}{\partial T}$.
\end{itemize}
If $f\in\C[X,T]$, then
\begin{itemize}
	\item $\cV(f):=\{\, (x,t)\in\C^2 \mid f(x,t)=0\,\}$
\end{itemize}
Moreover denote
\begin{itemize}
	\item $\R[T]_{\leq d}$ the vector space of polynomials of degree at most $d$,
	\item $H_d$ the set of $T$-real rooted polynomials of degree $d$ in $T$ and
	\item $H_d^+$ the set of those $f\in H_d$ that are strictly real rooted and define a smooth curve \footnote{That means the partial derivatives of $\widetilde{f}$ do not both vanish in any point of $\cV(f)$, where $\widetilde{f}$ denotes the square free core of $f$, i.e. the product of all its distinct prime factors.} $\cV(f)$,
	\item $S_d$ the set of symmetric $d\times d$ matrices over $\R[X]$,
	\item $\ch_d:S_d\to\R[X,T]:A\mapsto\det(T-A)$
\end{itemize}
In order to deduce the Helton-Vinnikov Theorem we will need some understanding of the relation between the degree of the entries of a matrix in $S_d$ and the degree of the coefficients of its characteristic polynomial, which we will gain in Section~\ref{sec:refine}. To this end we are going to introduce some further degree restrictions to the above defined objects.
\begin{itemize}
	\item $\R[X,T]_{(k,d)}:=\{\, \sum_{i=0}^{d}a_iT^i \in\R[X,T] \mid a_i\in\R[X],\, \deg a_i\leq k(d-i)\,\}$
	\item $H_{k,d}:=H_d\cap\R[X,T]_{(k,d)}$
	\item $H_{k,d}^+:=H_d^+\cap\R[X,T]_{(k,d)}$
	\item $S_{k,d}:=\{\, (a_{ij})_{i,j}\in S_d \mid \deg a_{ij}\leq k\,\}$
	\item $\ch_{k,d}:=\ch_d\mid_{S_{k,d}}$
\end{itemize}
The relation between $\R[X,T]_{(k,d)}$ and $S_{k,d}$ will become clearer in Lemma~\ref{lem:proper} (1).

\begin{rem}\label{rem:infinity}
	We will view $\R$ embedded into the real projective line $\P^1(\R)=\R\cup \{\infty\}$ and $X^{-1}$ as a continuous function on $U:=\P^1(\R)\setminus \{0\}$. Accordingly we extend the notion of polynomials that are (strictly) $T$-real rooted on $M$ also to elements of $\R[X^{-1},T]$ and subsets $M$ of $U$ in the obvious manner. Again we will omit ``on $M$'' if $M=U$.
\end{rem}

\begin{dfn}\label{defi:unimodular}
	Let $R$ be a commutative ring and $M$ an $R$-module. We call a symmetric bilinear form $\be:M\times M\to R$
	\begin{enumerate}[(1)]
		\item \emph{regular}, if the induced map $M\to M^*$ is injective and
		\item \emph{unimodular}, if the induced map $M\to M^*$ is an isomorphism.
	\end{enumerate}
	where $M^*$ denotes the dual module of $M$.
\end{dfn}
\begin{rem}\label{rem:unimodular}
	If in the situation of the previous definition $R$ is an integral domain, $M$ is finitely generated and torsion free, then surjectivity of $M\to M^*$ already implies injectivity, which is clear, if $R$ is a field, and can otherwise be reduced to that case by localizing at the zero ideal.
\end{rem}

\begin{dfn}
	For a Noetherian integral domain $R$ with field of fractions $K$ a \emph{fractional $R$-ideal} is a finitely generated $R$-submodule of $K$. Denote $\cI(R)$ the set of nonzero fractional ideals of $R$. A \emph{principal $R$-ideal} is a fractional $R$-ideal of the form $(a):=(a)_R:=Ra$ for some $a\in K$.
\end{dfn}


\begin{dfn}\label{defi:hermite}
	Let $K$ be a field, $f\in K[T]$ monic and $L:=K[T]/(f)=K[\al]$, where $\al=\overline{T}+(f)$. The representing matrix of the symmetric bilinear form defined via the trace form of $L$ over $K$
	\begin{align*}
		L\times L &\to K\\
		(a,b)&\to Tr_{L/K}(ab)
	\end{align*}
	with respect to the basis $1,\al,\dots,\al^{d-1}$, is called the \emph{Hermite matrix of $f$}.
\end{dfn}
\begin{rem}\label{rem:trace}
	Assume we are in the situation of the previous definition.
	\begin{enumerate}[(1)]
		\item If $K$ is the field of fractions of some integral domain $R$ and $f\in R[T]$ then $\Tr_{L/K}(R[\al])\subseteq R$ and therefore the entries of the Hermite matrix of $f$ lie in $R$.
		\item If $f=\prod_{i=1}^d(T-\la_i)\in K[T]$ all $\la_i$ lie in $K$ and are pairwise distinct, then $L=K[T]/(f)$ is isomorphic to $K^d$ via the Vandermonde map
			\begin{align*}
				V:L&\to K^d\\
				g+(f)&\mapsto (g(\la_i))_{i}
			\end{align*}
			and for any $\overline{g}=g+(f)\in L$ the endomorphism $\mu_{\overline{g}}$ of $L$ that is given by multiplication by $\overline{g}$ becomes diagonal with respect to the standard basis of $K^d$, when identified with $L$ via $V$. More precisely $V\circ\mu_{\overline{g}}\circ V^{-1}$ is simply given by multiplication of the $i$-th entry of a vector in $K^d$ by $g(\la_i)$. The trace can therefore be written as $\Tr_{L/K}(\overline{g})=\sum_{i=1}^dg(\la_i)$. Note that this last equality also holds, if not all the roots of $f$ lie in $K$, because we can reduce it to that case by extending scalars to some field that contains all the roots.
	\end{enumerate}
\end{rem}

The content of this note is based in an essential manner on the following very classical result on the Hermite matrix.

\begin{lem}\label{lem:hermite}
	If $f\in\R[T]$ is strictly real rooted, then the Hermite matrix of $f$ is positive definite.
\end{lem}
\begin{proof}
	It follows immediately from the considerations in Remark \ref{rem:trace}. Let $f=\prod_{i=1}^d(T-\la_i)$ where the $\la_i\in\R$ are pairwise distinct. Then
	\[\Tr_{L/\R}(\overline{g}^2)=\sum_{i=1}^dg(\la_i)^2\geq 0\]
	for $L=\R[T]/(f)$ and $\overline{g}=g+(f)\in L$ and ``$=0$'' holds if and only if $\overline{g}=0$.
\end{proof}

\begin{ass}\label{ass:RSKL}
	Let $R$ be a principal ideal domain with field of fractions $K$, $f\in R[T]$ monic, irreducible and separable of degree $d$. As above we define $L:=K[T]/(f)=K[\al]$ and $S:=R[\al]$, where $\al:=T+(f)\in L$. Moreover denote $\Tr:=\Tr_{L/K}$ the trace form of $L$ over $K$. $1,\al,\dots,\al^{d-1}$ form a basis of the $R$-module $S$ and the $K$-vector space $L$, referred to as the \emph{standard basis}.
\end{ass}

\begin{lem}\label{lem:fractional_ideal}
	Let Assumption \ref{ass:RSKL} hold. Moreover assume $I\in\cI(S)$ and denote $\mu_{\al}:I\to I:x\mapsto \al x$ multiplication by $\al$, viewed as an endomorphism of the $R$-module $I$. Then
	\begin{enumerate}[(1)]
		\item $I$ is free of rank $d$ as an $R$-module and
		\item the characteristic polynomial of $\mu_{\al}$ is $f$.
	\end{enumerate}
\end{lem}
\begin{proof}
	(1) Since $I$ is finitely generated over $S$ and $S$ is finitely generated over $R$, also $I$ is finitely generated over $R$. And since $I$ is torsion free and $R$ is a principal ideal domain, it is also free. The rank is at most $d$, since $L$ has $K$-dimension $d$. The rank is at least $d$, since $S$ is of rank $d$ over $R$ and we can embed $S$ into $I$ by $a\mapsto ax$ for any nonzero $x\in I$.

	(2) Multiplication by $\al$ as an endomorphism of $L$ has $f$ as its characteristic polynomial, as can easily be checked, e.g. by looking at its representing matrix with respect to the standard basis, which is also called the \emph{companion matrix of $f$}. By (1) every $\R[X]$-basis of $I$ is an $\R(X)$-basis of $L$ and characteristic polynomials are basis-independent.
\end{proof}

The following somewhat technical lemma will be very useful in the construction of certain unimodular bilinear forms and thus an essential ingredient in the proof of Theorem~\ref{thm:realrooted}. The proof of (1) is sometimes attributed to Euler and the way it is written here it is basically taken from Lang's book \citep[Chapter V, Proposition 5.5]{Lang02}. We include it here for selfcontainedness.

\begin{lem}\label{lem:trace}
	Let Assumption~\ref{ass:RSKL} hold. Then
	\begin{enumerate}[(1)]
		\item The bilinear form
			\begin{align*}
				\si:S\times S&\to R\\
				(a,b)&\mapsto \Tr\left(\frac{ab}{f'(\al)}\right)
			\end{align*}
			is welldefined and unimodular
		\item If $I\in\cI(S)$ and $c\in L$ such that $I^2=\left(\frac{c}{f'(\al)}  \right)$ then also
			\begin{align*}
				\be:I\times I&\to R\\
				(a,b)&\mapsto \Tr\left(\frac{ab}{c}\right)
			\end{align*}
			is welldefined and unimodular
	\end{enumerate}
\end{lem}

\begin{proof}
	(1) First of all notice that $f'(\al)\in L$ is nonzero, since $f$ is separable. Let $f=\prod_{i=1}^{d}(T-\la_i)$, where the $\la_i$ lie in some field extension of $K$. Then by Remark \ref{rem:trace} for $h\in K[T]$ we have
	\[\Tr \left( h(\al) \right)=\sum_{i}h(\la_i)\]
	
	Define $b:=\frac{f(X)-f(Y)}{X-Y}\in K[X,Y]$ and write $b=\sum_{k=0}^{d-1}\be_kY^k$, where the $\be_k\in K[X]$. Check that for $1\leq i,j\leq d$ we get
	\[b(\la_i,\la_j)=\de_{ij}f'(\la_i).\]
	And thus for $0\leq  \ell\leq d-1$ and $1\leq j\leq d$ we have
	\[\la_j^\ell=\sum_i \frac{\la_i^{\ell}}{f'(\la_i)}b(\la_i,\la_j)\]
	and since both sides of the equation are polynomial expressions in $\la_j$ of degree at most $d-1$ coinciding in the $d$ distinct points $\la_1,\dots,\la_d$ we therefore get
	\begin{align*}
		Y^{\ell}=\sum_{i=1}^d \frac{\la_i^{\ell}}{f'(\la_i)}b(\la_i,Y)=\sum_{i=1}^d\frac{\la_i^{\ell}}{f'(\la_i)}\sum_{k=0}^{d-1}\be_k(\la_i)&Y^k=\\
		=\sum_k \left( \sum_i \frac{\la_i^{\ell}\be_k(\la_i)}{f'(\la_i)} \right)Y^k=\sum_k \Tr\left(\frac{\al^{\ell}\be_k(\al)}{f'(\al)}  \right)&Y^k
	\end{align*}
	Comparing the coefficients of $Y^k$ we thus get for $0\leq k,\ell<d$
	\[\si(\al^{\ell},\be_k(\al))=\Tr\left( \frac{\al^{\ell}\be_k(\al)}{f'(\al)} \right)=\de_{\ell k}\]
	And as a special case
	\[\Tr\left( \frac{\al^{\ell}}{f'(\al)} \right)=\de_{\ell(d-1)}\]
	since $\be_{d-1}=1$.
	From this follows that if $a,b\in S$ then $\si(a,b)=\si(ab,1)$ is the coefficient of $\al^{d-1}$ in a representation of $ab\in S$ with respect to the standard basis and therefore lies in $R$. So $\si$ is welldefined, i.e. takes its values in $R$. We also see from the above, that $\si$ is unimodular, since $\be_0(\al),\dots,\be_{d-1}(\al)\in S$ form the dual basis of the standard basis with respect to $\si$.
	\medskip
	
	(2) Now fix an $R$-basis $e_1,\dots,e_d$ of $S$ and its dual basis $e_1^*,\dots, e_d^*\in S$, i.e. $\si(e_i,e_j^*)=\de_{ij}$, for example those from (1).

	We see easily that also $\be$ is welldefined: Since $I^2\subseteq(\frac{c}{f'(\al)})$ for any $a,b\in I$ there is $d\in S$ such that $ab=\frac{dc}{f'(\al)}$. But then
	\[\be(a,b)=\Tr\left(\frac{ab}{c}\right)=\Tr\left(\frac{d}{f'(\al)}\right)=\si(d,1)\in R.\]

	In order to check that $\be$ is unimodular we take an arbitrary module homomorphism $\nu:I\to R$ and show that $\nu=\be(a,\, .\,)$ for some $a\in I$, i.e. the induced homomorphism $I\to I^*$ is surjective and therefore bijective by Remark \ref{rem:unimodular}. We will view $\si$ and $\be$ naturally as bilinear forms $L\times L\to K$. We can also extend $\nu$ uniquely to a functional $L\to K$, since $I$ generates $L$ as a $K$-vector space, see Lemma~\ref{lem:fractional_ideal}. We can represent it in the usual manner via the dual basis in the following way:
	\begin{align*}
		\nu=&\si(r,.)\text{, where}\\
		r:=&\sum_{i=1}^d\nu(e_i)e_i^*
	\end{align*}
	We claim that $a:=\frac{r c}{f'(\al)}$ lies in $I$ and $\nu=\be(a,.)$. The latter is easy since
	\[\nu(b)=\si(r,b)=\Tr\left(\frac{r b}{f'(\al)}\right)=\Tr\left(\frac{r c}{f'(\al)}\frac{b}{c}  \right)=\Tr\left(\frac{ab}{c}\right)=\be(a,b)\]
	for all $b\in I$ as desired.
	
	In order to show that $a\in I$, we first notice that $rI\subseteq S$: For any $b\in I$ we can represent $rb$ again with the dual basis $e_1^*,\dots,e_d^*$
	\[rb=\sum_{i=1}^d\al_ie_i^*\]
	where the coefficients $\al_i$ lie in $R$ since
	\[\al_i=\si(rb,e_i)=\si(r,be_i)=\nu(be_i)\in R\]
	because $e_i\in S$ and therefore $be_i\in I$. Consequently $rb\in S$. Now we use the other inclusion $I^2\supseteq (\frac{c}{f'(\al)})$ to find $x_i,y_i\in I$ such that
	\[\frac{c}{f'(\al)}=\sum_i x_iy_i\] and therefore
	\[a=\frac{r c}{f'(\al)}=\sum_i\underbrace{r x_i}_{\in S}y_i\in I.\]
	\medskip
\end{proof}

Another elementary however not completely trivial but important ingredient is the following special feature of unimodular forms over univariate polynomial rings.
\begin{thm}\label{thm:diagonalization}
	Let $K$ be a field of characteristic not equal to $2$ and $M$ a free $K[X]$-module of rank $n$. Then any unimodular $K[X]$-bilinear form $\be$ on $M$ admits an orthogonal basis $q_1,\dots,q_n$. Moreover for every such orthogonal basis we have $\be(q_i,q_i)\in K^{\times}$.
\end{thm}

\begin{proof}
	See for example Scharlau's book \citep[Chapter 6, Theorem~3.3]{Scharlau11_book}.
\end{proof}

\section{Representations of Smooth Curves}\label{sec:smoothrealrooted}

In this section we will undertake the most important step in order to prove Theorem~\ref{thm:realrooted}. Namely we will first show that any strictly $T$-real rooted polynomial $f\in\R[X,T]$ is the characteristic polynomial of a symmetric matrix over $\R[X]$, if we assume, that the curve it defines is smooth. In fact, this assumption can be avoided, if one works with the normalization of the curve, instead of the curve itself. However, in order to keep the setup and language as elementary as possible, we are going to reduce the general case to this smooth case in the following two sections.

We now state two basic facts about Dedekind domains that we will need and which can be found in one way or the other or at least obtained from most books on algebraic geometry and commutative algebra. See for example \citep[Chapter I, Section 6]{Hartshorne77} and \citep[Section 11.4]{Eisenbud95}.

\begin{dfn}
	A \emph{Dedekind domain} is an integral domain $R$, that is one dimensional (i.e. nonzero prime ideals are maximal) and integrally closed in its field of fractions.
\end{dfn}

\begin{prop}\label{prop:smoothisdedekind}
	The coordinate ring of a smooth and irreducible affine algebraic curve is a Dedekind domain.
\end{prop}

\begin{prop}\label{prop:idealfactorization}
	If $R$ is a Dedekind domain, then the set of nonzero fractional $R$-ideals $\cI(R)$ forms an abelian group via the usual ideal multiplication and it is freely generated by its nonzero prime ideals, i.e. every fractional ideal is admits a unique factorization into prime ideals.
\end{prop}

There is one classical and standard but nontrivial result about class groups of smooth curves over the complex numbers, that we are going to use without proof namely Theorem~\ref{thm:divisible}. However there don't seem to exist many explicit references for it.
\begin{thm}\label{thm:divisible}
	Let $S:=\C[\cC]$ be the coordinate ring of a smooth and irreducible affine curve $\cC$ over $\C$. Then the \emph{ideal class group} $\Cl(S)$ of $S$, i.e. $\cI(S)$ modulo the subgroup of principal ideals, is divisible. In other words for any nonzero fractional $S$-ideal $J_0$ and natural number $n\in\N_{>0}$, there exists another fractional $S$-ideal $J_1$ and an element $e$ of the function field, such that $J_0=eJ_1^n$.
\end{thm}
Since it is the central result needed for our proof of Theorem~\ref{thm:smoothrealrooted}, let us give at least some explanation of this fact: It is folklore that there is an isomorphism between the group of divisors $\Div(\cC)$ (i.e. formal linear combinations of points on $\cC$) and the group of fractional ideals $\cI(S)$ of $S$ given by
\begin{align*}
	\Div(\cC)&\to \cI(S)\\
	\sum_i \nu_i p_i&\mapsto \prod_i I(p_i)^{-\nu_i}
\end{align*}
where the $p_i\in\cC$, $\nu_i\in\Z$ and for $p\in\cC$ the vanishing ideal $I(p)$ of $p$, i.e. the kernel of the evaluation map $S\to\C$ at $p$. This isomorphism induces an isomorphism on the respective class groups, i.e. $\Pic(\cC)$ the Picard group (or divisor class group, i.e. $\Div(\cC)$ modulo the subgroup of principal divisors) of the curve and the ideal class group $\Cl(S)$ of its coordinate ring.

If $\widehat{\cC}$ is the associated complete curve, i.e. for example (the normalization of) the projective closure of $\cC$ or simply any complete smooth curve having $\C(\cC)$ as its function field, then $\Pic(\cC)$ is the quotient of $\Pic^0(\widehat{\cC})$ (the degree $0$ part of $\Pic(\widehat{\cC})$) by the subgroup generated by the points at infinity, i.e. by $\widehat{\cC}\setminus \cC$.

Now $\Pic^0(\widehat{\cC})$ is a classical and widely studied object. The Abel-Jacobi Theorem states that it is isomorphic to the Jacobian Variety of $\widehat{\cC}$ which again can be viewed as a complex torus, more precisely the additive group $\C^g/\La$, where $g$ is the genus of $\widehat{\cC}$ and $\La$ is a sublattice of $\C^g$ of rank $2g$. From this theorem the divisibility follows immediately, since $\C^g$ is a divisible group. This however involves nonalgebraic methods, namely the Abel-Jacobi map, i.e. the isomorphism that identifies the degree zero component of the Picard group with the Jacobian variety.

A completely algebraic treatment of this topic has been done by Weil in \citep{weil46}. There the Jacobian of a complete smooth curve is defined as a certain \emph{abelian variety}, a complete group variety, attached to the curve. It is shown that $\Pic^0(\widehat{\cC})$ is isomorphic to the Jacobian, also in this algebraic setup. See \citep[Chapitre V, Théorème 19]{weil46}. Multiplication by some positive natural number $n$ in such groups is dominant and since they are complete, it is even surjective. That means the group is $n$-divisible.

See also \citep{Milne08} or other books on abelian varieties (e.g. by Mumford or Lang) for more details on this topic. A direct proof of the divisibility of $\Pic^0(\widehat{\cC})$, that does not use the construction of the Jacobian can be found in \citep{Frey79}.

\medskip
Using these facts, we can now proceed with the central step in our endeavor.

\begin{thm}\label{thm:smoothrealrooted}
	$H_d^+\subseteq \im \ch_d$. In other words every strictly $T$-real rooted polynomial $f\in\R[X,T]$ that defines a smooth curve $\cV(f)$ admits a symmetric $T$-spectral determinantal representation.
\end{thm}

\begin{proof}
	We are going to find a suitable free $\R[X]$-module $I$ of rank $d$ and an endomorphism $\mu$ of $I$ with $f$ as its characteristic polynomial. Moreover, we will equip this module with a bilinear form $\be$ with the following two properties:
	\begin{itemize}
		\item $\mu$ is selfadjoint with respect to $\be$ and
		\item $\be$ admits an orthonormal basis.
	\end{itemize}
	We can then take $A$ to be the representing matrix of $\mu$ with respect to this orthonormal basis, which is symmetric and whose characteristic polynomial is also $f$.

	\medskip
	First of all, we can assume, that $f$ is irreducible, otherwise we work with each irreducible factor separately and form a block diagonal matrix composed of all their spectral determinantal representations, yielding a spectral determinantal representation of $f$.

	\medskip
	Define the $\R[X]$-algebra $S:=\R[X,T]/(f)$, the real coordinate ring of $\cV(f)$ and $L$ its function field, i.e. the field of fractions of $S$. Write $\Tr$ for the trace form $\Tr_{L/\R(X)}$ of $L$ over $\R(X)$. Moreover let $S_{\C}:=\C[X,T]/(f)$ be the complex coordinate ring with field of fractions $L_{\C}$. Note that $f$ is indeed irreducible over $\C$:
	
	Otherwise it would have to factor into pairs of complex conjugate irreducible monic polynomials $f=\prod_{i=1}^{n}\overline{g_i}g_i$. Since $\overline{g_i}g_i\in\R[X,T]$ and $f$ is irreducible over $\R$, we have $n=1$,  But then the set of real points of $\cV(f)$ would be the intersection of $\cV(g_1)$ and $\cV(\overline{g_1})$, which is finite, contradicting our main assumption on $f$.

	\medskip
	The central step in this proof will be to show that there exists a fractional $S$-ideal $I\in\cI(S)$ and a nonzero sum of squares $c\in L$ such that $I^2=\left( \frac{c}{f'(\al)} \right)$, where
	\[\al:=T+(f)\in S.\]
	Then we can define the unimodular bilinear form
	\begin{align*}
		\be:I\times I&\to \R[X]\\
		(a,b)&\mapsto \Tr\left(\frac{ab}{c}\right)
	\end{align*}
	as in Lemma~\ref{lem:trace}. By Lemma~\ref{lem:fractional_ideal} $I$ is a free $\R[X]$-module of rank $d$ and the endomorphism $\mu:I\to I:x\mapsto \al x$ has $f$ as its characteristic polynomial. Obviously $\mu$ is self-adjoint with respect to $\be$, because for $a,b\in I$
	\[\be(\mu a, b)=\Tr\left(\frac{(\al a) b}{c}\right)=\Tr\left( \frac{a(\al b)}{c} \right)=\be(a,\mu b)\]

	We can conclude from Theorem~\ref{thm:diagonalization} that $\be$ admits an orthogonal basis. If it even admits an \textbf{orthonormal} basis, we can take $A\in\R[X]^{d\times d}$ as announced to be the representing matrix of $\mu$ with respect to this orthonormal basis. Then $A$ is symmetric and
	\[f=\det(T-A)\]
	as desired. What is now left to show is
	\begin{enumerate}[(1)]
		\item how to obtain this fractional ideal and
		\item that we can find an orthonormal basis for $\be$.
	\end{enumerate}

	\medskip
	(1) Since $\cV(f)$ is smooth, $S_{\C}$ and $S$ are Dedekind domains by Proposition \ref{prop:smoothisdedekind}. We look at the prime ideal decomposition of $\left( f'(\al) \right)$ in $\cI(S_{\C})$ (see Proposition \ref{prop:idealfactorization}). So let
	\[\left(f'(\al)  \right)_{S_{\C}}=\prod_{i=1}^sP_i\prod_{j=1}^t \overline{Q_j} Q_j,\]
	where the $P_i$ are the vanishing ideals of real points and the $Q_j$ those of nonreal points on $\cV(f)$. Note that the nonreal points come in conjugate pairs, since $f'(\al)$ is defined over $\R$, i.e. lies in $S$. Since by assumption for all $x\in \R$ the roots of $f(x,T)$ are all simple, there is no real point $p\in\cV(f)$ such that $f'(p)=0$ and therefore $s$ must be $0$. Taking $J_0:=\prod_{j=1}^t Q_j$ we have
	\[\left( f'(\al) \right)_{S_{\C}}=\overline{J_0}J_0.\]
	Now we use the fact that the class group $\Cl(S_{\C})$ is divisible by Theorem~\ref{thm:divisible}, so there exist $e\in L_{\C}$ and a fractional $S_{\C}$-ideal $J_1$ such that $J_0=eJ_1^2$. Therefore
	\[\left( f'(\al) \right)_{S_{\C}}=\overline{J_0}J_0=\overline{eJ_1^2}eJ_1^2=\overline{e}e(\overline{J_1}J_1)^2.\]
	If $e=e_1 +ie_2$, where $e_1,e_2\in L$, then $c:=\overline{e}e=e_1^2+e_2^2$ is a sum of squares in $L$. So we get
	\[\left( \frac{c}{f'(\al)} \right)_{S_{\C}}=\left( \overline{J_1}J_1 \right)^{-2}\]
	is a square in $\cI(S_{\C})$. But then also $\left( \frac{c}{f'(\al)} \right)_S$ is a square in $\cI(S)$ \footnote{One might be tempted to say it is clear that $\left( \frac{c}{f'(\al)} \right)_S=\left( \overline{J_1}J_1\cap L \right)^{-2}$, however the map $\cI(S_{\C})\to\cI(S):I\mapsto I\cap L$ is \textbf{not} multiplicative, and upon a closer look the statement seems to be not completely obvious if one is not familiar with extensions of Dedekind domains and some geometry of curves.}: We look at the prime ideal decomposition of $\left( \frac{c}{f'(\al)} \right)_S$ in $\cI(S)$. So let
	\[\left( \frac{c}{f'(\al)} \right)_S=\prod_iP_i^{\nu_i}\prod_jQ_j^{\mu_j}\]
	where again the $P_i$ correspond to pairwise distinct real points and the $Q_j$ to pairwise distinct conjugate pairs of nonreal points. Then $S_{\C}P_i$ is prime and $S_{\C}Q_j=\overline{Q_j'}Q_j'$ for prime ideals $Q_j'$ in $S_{\C}$. Moreover the map
	\begin{align*}
		\cI(S)&\to\cI(S_{\C})\\
		I&\mapsto S_{\C}I
	\end{align*}
	is obviously multiplicative and therefore a group homomorphism. Combining these two observations yields
	\[\left( \frac{c}{f'(\al)} \right)_{S_{\C}}=\prod_i\left( S_{\C}P_i \right)^{\nu_i}\prod_j\left( \overline{Q_j'} \right)^{\mu_j}\prod_j\left( Q_j' \right)^{\mu_j}\]
	is the prime ideal decomposition of $\left( \frac{c}{f'(\al)} \right)_{S_{\C}}$ in $\cI(S_{\C})$. Since it is a square therein and the $S_{\C}P_i$, $\overline{Q_j'}$ and $Q_j'$ are all pairwise distinct all $\nu_i$ and $\mu_j$ must be even. That means $\left( \frac{c}{f'(\al)} \right)_S$ is a square in $\cI(S)$.

	\medskip
	(2) Now that we showed the existence of the fractional ideal $I$ as desired, we can define the unimodular bilinear form $\be$ as above. In order to see that we can find an orthonormal basis, we take an orthogonal basis $q_1,\dots,q_d\in I$ where $\be(q_i,q_i)\in\R^{\times}$ as in Theorem~\ref{thm:diagonalization}. We are going to use Lemma~\ref{lem:hermite} to see that these numbers are positive. So we fix $i$ and set $q:=q_i$. We have that $c$ and thus also its inverse is a sum of (two) squares in $L$, so there exist $c_1,c_2\in L$ such that
	\[\frac{1}{c}=c_1^2+c_2^2\]
	For $j\in \{1,2\}$ let $v_j$ be the coordinate vector of $qc_j$ with respect to the $\R(X)$-basis $1,\al,\al^2,\dots,\al^{d-1}$ of $L$ and denote by $H\in S_d$ the Hermite matrix of $f$. Then
	\[\be(q,q)=\Tr\left( \frac{q^2}{c} \right)=\Tr\left( (qc_1)^2 \right)+\Tr\left( (qc_2)^2 \right)=v_1^{\intercal}Hv_1+v_2^{\intercal}Hv_2\]
	The left hand side is constant and can therefore be obtained by evaluating the right hand side at any $x\in\R$ where $v_1$ and $v_2$ have no pole. Since $H(x)$ is the Hermite matrix of $f(x,T)$, it is positive definite by Lemma~\ref{lem:hermite}, i.e.
	\[\be(q,q)=v_1(x)^{\intercal}H(x)v_1(x)+v_2(x)^{\intercal}H(x)v_2(x)\geq 0.\]
	By rescaling the $q_i$ with $(\be(q_i,q_i))^{-\frac{1}{2}}\in\R$ we can thus assume, that in fact $q_1,\dots,q_d$ is already an orthonormal basis.
\end{proof}

\section{Smooth approximation of real rooted polynomials}\label{sec:approx}

In order to reduce the general case of not necessarily strictly real rooted polynomials or those defining singular curves to the smooth and strictly real rooted case, we are going to show that $H_d^+$ is dense in $H_d$. The idea behind the proof of this fact is that almost any perturbation of a real rooted polynomial yields a polynomial, that defines a smooth curve, however we have to be careful to not lose the real rootedness while perturbing. This property is somehow robust to perturbation, as long as the real roots are far enough apart. So the first step of our perturbation is to get rid of possibly multiple real roots and make the polynomial strictly real rooted. This is done in Lemma~\ref{lem:multiplicity_reduction} by some ``multiplicity reduction operator'' $P_a$ and is basically the same method as used by Nuij in \citep{Nuij68} to approximate hyperbolic polynomials by strictly hyperbolic ones.

\begin{dfn}\label{defi:PandQ}
	For a commutative ring $R$ and element $a\in R$ we define
	\begin{align*}
		P_a: R[T]&\to R[T]\\
		g&\mapsto g+ag'
	\end{align*}
	If $b\in R^{\times}$ and $d\in\N$ we moreover define the ``scaling operator''
	\begin{align*}
		Q_{d,b}:R[T]_{\leq d}&\to R[T]_{\leq d}\\
		g&\mapsto b^{-d}g(bT)
	\end{align*}
\end{dfn}

\begin{lem}\label{lem:multiplicity_reduction}
	Let $f\in\R[T]$ be real rooted. Then for any nonzero $\ep\in\R$ we have
\begin{enumerate}[(1)]
	\item $P_{\ep}f=f+\ep f'$ is real rooted.
	\item $P_{\ep}$ reduces the multiplicity of roots of $f$, i.e. for each root $\la$ of $f$ of multiplicity $\mu>0$, the multiplicity of $\la$ as a zero of $P_{\ep}f$ is $\mu-1$.
	\item $P_{\ep}$ does not produce new multiple roots, i.e. each multiple root of $P_{\ep}f$ was already a multiple root of $f$ (thus of multiplicity one higher by (2)).
\end{enumerate}
\end{lem}
\begin{proof}
	Let $f=\prod_{i=1}^n(T-\la_i)^{\mu_i}$, where the $\la_1<\dots<\la_n\in\R$ are the pairwise distinct zeros of $f$. We look at the zeros of
	\[g:=\frac{P_{\ep}f}{ \ep f}=\ep^{-1} +\frac{f'}{f}=\ep^{-1}+\sum_{i=1}^n \frac{\mu_i}{T-\la_i}\]
	which are a subset of the zeros of $P_{\ep}f$. At each $\la_i$ we have a simple pole of $g$ with sign change from negative to positive. Therefore there is a zero $\ga_i\in (\la_i,\la_{i+1})$ for $i<n$. And since $g(t)\to 0$ for $t\to\pm\infty$ there is another zero $\ga_n$ either in $(-\infty,\la_1)$, if $\ep>0$ or in $(\la_n,\infty)$ if $\ep<0$.
	
	So $\prod_{i=1}^n(T-\ga_i)$ divides $P_{\ep}f$. But as can easily be seen also $\prod_{i=1}^n(T-\la_i)^{\mu_i-1}$ divides $P_{\ep}f$ and comparing degrees yields
	\[P_{\ep}f=\prod_{i=1}^n(T-\ga_i)(T-\la_i)^{\mu_i-1}\]
	from which (1)-(3) follows immediately.
\end{proof}

\begin{rem}\label{rem:realrooted_vs_hyperbolic}
The following paragraphs are a little excursion towards some geometric difference between our notion of real rooted polynomials and hyperbolic polynomials. It is supposed to serve as a motivation for some technical steps and explain what we mean by roots being far enough apart to assure stability of real rootedness under perturbation. Since it is however not crucial for the proofs, some parts will stay somewhat vague.

Let $f\in\R[X,T]$ be $T$-real rooted of degree $d$ in $T$. Let $\cC:=\cV(f)$ be the affine curve defined by $f$. Denote $\pi$ the projection onto the affine line $\A^1$
\begin{align*}
	\pi:\cC&\to \A^1\\
	(x,t)&\mapsto x
\end{align*}
The real rootedness means that real points $x\in\A^1(\R)$ have real fibers $\pi^{-1}(x)\subseteq\cC(\R)$. Then $f$ is strictly real rooted if and only if $\pi$ is unramified over all real points $x\in\A^1(\R)$, i.e. $\pi^{-1}(x)$ does not contain any multiple points.

If now in addition $f\in\R[X,T]_{(1,d)}$, then its homogenization $F:=Y^df\left( \frac{X}{Y},\frac{T}{Y} \right)\in\R[X,Y,T]$ is of degree $d$ and \textbf{hyperbolic} with respect to $(0,0,1)$, i.e. the extension $\widehat{\pi}$ of $\pi$ to the projective closure $\widehat{\cC}:=\cV_+(F)\subseteq\P^2$ of $\cC$
\begin{align*}
	\widehat{\pi}: \widehat{\cC}&\to \P^1\\
	[x:y:t]&\mapsto [x:y]
\end{align*}
has real fibers over real points. $F$ is \emph{strictly hyperbolic}, if in addition $\widehat{\pi}$ is unramified over all real points in $\P^1(\R)$, i.e. if $f$ is strictly real rooted and in addition has ``only simple roots at infinity''. In our setting we consider polynomials that can have arbitrary degree in $X$, which makes it a bit more technical to analyze the behavior ``at infinity''.


If now $f\in H_{k,d}\setminus H_{1,d}$ for some $k>1$, then in general the roots of $f(x,T)$ will grow like $x^k$ when $x$ gets large, so they all pass through the point
\[\lim_{x\to\infty}[x:1:x^k]=\lim_{x\to\infty}[x^{-k+1}:x^{-k}:1]=[0:0:1]\]
when viewed in $\widehat{\cC}$. But therefore $\widehat{\pi}$ will be totally ramified in this point, which is mapped to the real point $\infty\in\P^1(\R)$. So in this sense $f$ generally has multiple roots at infinity if $f\in\R[X,T]_{(k,d)}$ for $k>1$. We get a more suitable notion of simple roots at infinity in our setup, if we look at them after rescaling, namely at the zeros of $Q_{d,X^k}f\in\R[X^{-1},T]$, see Lemma~\ref{lem:PandQ}. For a generic polynomial $f\in\R[X,T]_{(k,d)}$ the zeros of $Q_{d,X^k}f(\infty,T)$ will all be simple. Denote $g\in\R[Y,T]$ the polynomial that we get by replacing $X^{-1}$ in $Q_{d,X^k}f$ by $Y$. Then we can glue together the two affine curves $\cV(f)$ and $\cV(g)$ via the isomorphism
\begin{align*}
	\ph:\cV(f)\setminus \cV(X)&\to\cV(g)\setminus \cV(Y)\\
	(x,t)&\mapsto (x^{-1},x^{-k}t)
\end{align*}
to get a curve $\widetilde{\cC}$ which ``distinguishes the points at infinity'' i.e. the projection $\pi: \widetilde{\cC}\to\P^1$ is unramified over $\infty$, if $Q_{d,X^k}f(\infty,T)=g(0,T)$ has only simple roots. So what we will take as a reasonable strengthening of the strict real rootedness of $f$ is that $\widetilde{\pi}$ is unramified over all of $\P^1(\R)$. Then small perturbations of $f$ will still be real rooted. See Lemma~\ref{lem:approx}.

Note that in the case $f\in\R[X,T]_{(1,d)}$ this curve $\widetilde{\cC}$ coincides with the projective closure $\widehat{\cC}$, where $V(f)$ and $V(g)$ are the two affine charts that lie over $\P^1\setminus \{\infty\}$ and $\P^1\setminus \{0\}$, respectively. So in this special case our additional condition on $f$, is simply \textbf{strict hyperbolicity} of $F$, as described above.

Another way of viewing simplicity of roots of $Q_{d,X^k}f(\infty,T)$ is that the mutual distances between the roots of $f(x,T)$ also grow like $x^k$ for large $x$.
\end{rem}

\begin{lem}\label{lem:PandQ}Let $R$ be a ring and $k,d\in\N$.
	\begin{enumerate}[(1)]
		\item Using the chain rule for differentiation we see that $P_aQ_{d,b}=Q_{d,b}P_{ab}$ for all $a\in R$ and $b\in R^{\times}$.
		\item If $a\in \R[X]$ is of degree at most $k$, then $\R[X,T]_{(k,d)}$ is $P_a$-invariant.
		\item $Q_{d,X^k}\R[X,T]_{(k,d)}\subseteq\R[X^{-1},T]$, where we view $b=X^k$ as a unit in the ring $\R(X)$.
		\item If $f\in\R[X,T]_{(k,d)}$, then $f$ is strictly $T$-real rooted on $\R\setminus \{0\}$ if and only if $Q_{d,X^k}f$ is.
	\end{enumerate}
\end{lem}

\begin{proof}
	These are all straight forward computations: (1) For $f\in R[T]$, $a\in R$ and $b\in R^{\times}$ we have
	\begin{align*}
		P_aQ_{d,b}f&=P_a\left( b^{-d}f(bT) \right)=b^{-d}f(bT)+a b^{-d}\frac{\partial }{\partial T}\left( f(bT) \right)=\\
		&=b^{-d}f(bT) + ab^{-d+1} \left( \frac{\partial f}{\partial T} \right)(bT)=b^{-d}\left( f+ab \frac{\partial f}{\partial T} \right)(bT)=\\
		&=b^{-d}(P_{ab}f)(bT)=Q_{d,b}P_{ab}f
	\end{align*}

	(2) If $f\in\R[X,T]_{(k,d)}$, i.e. $f=\sum_{i=0}^d a_iT^i$ with $a_i\in\R[X]$ of degree at most $k(d-i)$, then $f'=\sum_{i=0}^{d-1}b_iT^i$, where $b_i=(i+1)a_{i+1}$ is of degree at most $k(d-i-1)$ and therefore, if $a\in R[X]$ is of degree at most $k$, then $ab_i$ is of degree at most $k(d-i)$ and thus $af'\in\R[X,T]_{(k,d)}$.

	(3) If $f\in\R[X,T]_{(k,d)}$ is as in (2), then $Q_{d,X^k}f=\sum_{i=0}^da_iX^{k(i-d)}T^i$ and therefore all coefficients lie in $\R[X^{-1}]$.

	(4) Let $f\in\R[X,T]_{(k,d)}$. Then $Q_{d,X^k}$ does not change the coefficient of $T^d$ and for all $x\in\R$ the map $\la\mapsto x^k\la$ is a bijection between the roots of $Q_{d,X^k}f(x,T)$ and the roots of $f(x,T)$, under which realness and simplicity is preserved.
\end{proof}


\begin{lem}\label{lem:approx}
	For any $k,d\in\N$ the set $H_{k,d}^+$ is dense in $H_{k,d}$.
\end{lem}

\begin{proof}
	We will prove the density in a couple of intermediate steps. For this we define a descending chain of sets $M_0:=H_{k,d}\supseteq \dots\supseteq M_4$ such that $M_4\subseteq H_{k,d}^+$ and show density at each inclusion.
	\begin{itemize}
		\item $M_1:=\{\, f\in M_0 \mid f(0,T)\text{ is strictly real rooted}\,\}$
		\item $M_2:=\{\, f\in M_1 \mid f \text{ and }(Q_{d,X^k}f)(\infty,T)\text{ are strictly $T$-real rooted }\footnote{See Remark \ref{rem:infinity}\text{ for this notation.}}\,\}$
		\item $M_3:=\{\, f\in M_2 \mid \frac{\partial f}{\partial X}, \frac{\partial f}{\partial T}\text{ are coprime}\,\}$
		\item $M_4:=\{\, f\in M_3 \mid \cV(f)\text{ is smooth}\,\}$
	\end{itemize}
	As explained in the beginning of this section, the idea is to first approximate a $T$-real rooted polynomial $f\in\R[X,T]$ by some $g\in\R[X,T]$ for which the zeros of $g(x,T)$ stay far enough apart for all $x$. This assures, that under further perturbation, the real rootedness is not lost. Here ``far enough apart'' basically means that $g\in M_2$.

	Let $f\in M_0$. Then $f_{\ep}:= P_{\ep}^{d-1}f\in M_1$ for all nonzero $\ep\in\R$ by Lemma~\ref{lem:multiplicity_reduction}, since
	\[f_{\ep}(x,T)=\left( P_{\ep}^{d-1}f \right)(x,T)=P_{\ep_1}^{d-1}\left( f(x,T) \right)\]
	is strictly real rooted for all $x\in\R$, in particular for $x=0$. And by Lemma~\ref{lem:PandQ} $\R[X,T]_{(k,d)}$ is invariant under $P_{\ep}$.

	Let $f\in M_1$. Then $f_{\ep}:= P_{\ep X^k}^{d-1}f\in M_2$ for all nonzero $\ep\in\R$: Again by Lemma~\ref{lem:multiplicity_reduction} we have that
	\[f_{\ep}(x,T)=\left( P_{\ep X^k}^{d-1}f \right)(x,T)=P_{\ep x^k}^{d-1}\left( f(x,T) \right)\]
	is strictly real rooted for all nonzero $x\in\R$, as well as $f_{\ep}(0,T)=f(0,T)$. Therefore $f_{\ep}$ is strictly $T$-real rooted and again $f_{\ep}\in\R[X,T]_{(k,d)}$ by Lemma~\ref{lem:PandQ}. Moreover we can use the same Lemma to see that also
	\begin{align*}
		\left( Q_{d,X^k}f_{\ep} \right)(\infty,T)=\left( Q_{d,X^k}P_{\ep X^k}^{d-1}f \right)(\infty,T)&=\\
		=\left( P_{\ep}^{d-1}Q_{d,X^k}f \right)(\infty,T)=P_{\ep}^{d-1}\left( (Q_{d,X^k}f)(\infty,T) \right)&
	\end{align*}
	is strictly real rooted.

	Let $f\in M_2$. Then $f_{\ep}:=f+\ep T\in M_3$ for $0<\ep\in\R$ small enough: $\frac{\partial f_{\ep}}{\partial X}$ and $\frac{\partial f_{\ep}}{\partial T}$ are coprime for almost all $\ep$, because every prime factor of $\frac{\partial f_{\ep}}{\partial X}=\frac{\partial f}{\partial X}$ divides $\frac{\partial f_{\ep}}{\partial T}=\frac{\partial f}{\partial T}+\ep$ for at most one $\ep$. In order to see that $f_{\ep}$ is still in $M_2$ for small $\ep$, we check that the strict real rootedness of $f$ and $(Q_{d,X^k}f)(\infty,T)$ is an open condition. Fix a compact neighborhood $K_1=\P^1(\R)\setminus (-c,c)$ of $\infty$. Then $Q_{d,X^k}f$ is strictly $T$-real rooted on $K_1$ (see Lemma~\ref{lem:PandQ},(4)). Since the zeros of a monic polynomial depend continuously on its coefficients, we see that also for all $g$ in a neighborhood $U_1$ of $f$ we have that $Q_{d,X^k}g$ is strictly $T$-real rooted on $K_1$. Likewise we get a neighborhood $U_2$ of $f$, such that all $g\in U_2$ are strictly $T$-real rooted on $K_2:=[-c,c]$. But then for all $g\in U_1\cap U_2$ we have that $g$ is strictly $T$-real rooted on $K_1\setminus \{\infty\}\cup K_2=\R$ and $(Q_{d,X^k}g)(\infty,T)$ is strictly real rooted. That means the open neighborhood $U_1\cap U_2$ of $f$ lies in $M_2$.

		Let $f\in M_3$. Then $f_{\ep}:=f+\ep\in M_4$ for $\ep>0$ small enough: The condition on the coprimeness of the partial derivatives is trivially satisfied, because $f_{\ep}$ and $f$ differ only by a constant. Moreover we have seen, that $M_2$ is open. So we only need to check that $V(f_{\ep})$ is smooth for small $\ep>0$. But this is easy, since we just have to make sure that $f+\ep$ is nonzero on the finitely many common zeros of $\frac{\partial f}{\partial X}$ and $\frac{\partial f}{\partial T}$.
\end{proof}

\section{Refinement}\label{sec:refine}
In this section we are going to investigate properties of the map $\chi_{k,d}$ a bit further. On the one hand we will get a better understanding of how the degree of the entries of a matrix in $S_d$ is related to the degree on the coefficients of its characteristic polynomial. Namely the next lemma tells us, that a matrix in $S_d$ has degree at most $k$, if and only if its characteristic polynomial lies in $\R[X,T]_{(k,d)}$. On the other hand, we will see that $\chi_{k,d}$ is in a way well behaved with respect to approximations, in the sense that it is a proper map. Those two, at first sight seemingly unrelated properties, are in fact based on the same principle, namely that the size of the eigenvalues of a real symmetric matrix yields a bound on its coefficients.

\begin{lem}\label{lem:proper}
	Let $k,d\in\N$. Then
	\begin{enumerate}[(1)]
		\item $\ch_d^{-1}(\R[X,T]_{(k,d)})=S_{k,d}$
		\item $\ch_{k,d}:S_{k,d}\to\R[X,T]_{(k,d)}$ is proper
		\item $\im \ch_{k,d}$ is closed in $\R[X,T]_{(k,d)}$ 
	\end{enumerate}
\end{lem}
Here the words ``proper'' and ``closed'' refer to the Euclidean topology on finite dimensional $\R$-vector spaces. In the proof we will make use of the fact that any two norms induce the same topology and therefore we can conveniently choose what ``bounded'' means.

\begin{proof} Let us first note that one inclusion in (1), namely
	\[\ch_d(S_{k,d})\subseteq\R[X,T]_{(k,d)}\]
	is easy: Assume $A\in S_{k,d}$, i.e. its entries are of degree at most $k$. Then the coefficient of $T^i$ in $f:=\det(T-A)$ is a sum of $d-i$-minors of $A$ and thus of degree at most $k(d-i)$ in $X$, i.e. $f\in\R[X,T]_{(k,d)}$. This also makes the map in (2) welldefined.

	The next step is to prove the scalar version of (2), i.e. for $k=0$. From this we will deduce the case for general $k$, and also the other inclusion in (1).

	\medskip
	Let $\ch:=\ch_{0,d}:\Sym_d(\R)\to\R[T]_{\leq d}$. Obviously $\ch$ is continuous, because it is given by polynomials in the entries of the matrices. Assume we have a bounded set $F\subseteq\R[T]_{\leq d}$. We want to show that $\ch^{-1}(F)$ is also bounded. Since $\im\ch$ consists of only monic polynomials, we can assume that also $F$ consists only of monic polynomials. Therefore a bound on the coefficients of the elements of $F$ gives also a bound on their zeros which are the eigenvalues of the matrices in $\ch^{-1}(F)$. Moreover, the operator norm of a real symmetric matrix coincides with its maximal eigenvalue and thus we see that $\ch^{-1}(F)$ is bounded as soon as $F$ is bounded.

	\medskip
	Now let $k\in\N$ be arbitrary. The idea of the proof is the following: Boundedness in a finite dimensional vector space of polynomials in $X$ corresponds to a uniform bound on some proper compact interval, while a bound on the degree in $X$ corresponds to a uniform bound in a neighborhood of infinity, after rescaling in an appropriate manner.

	So first take $A\in S_d$ such that $f:=\det(T-A)\in\R[X,T]_{(k,d)}$. We want to show that the entries of $A$ have degree at most $k$. Therefore we look at
	\[g:=Q_{d,X^k}f=X^{-kd}\det(X^kT-A)=\det(T-X^{-k}A)\]
	By Lemma~\ref{lem:PandQ} we have $g\in\R[X^{-1},T]$. In order to see that already $B:=X^{-k}A\in\R[X^{-1}]^{d\times d}$ and therefore $A\in S_{k,d}$, we look at the set
	\[B_{(1,\infty)}:=\{\, B(x) \mid x\in(1,\infty)\,\}\subseteq \ch^{-1}(g_{(1,\infty)})\]
	where
	\[g_{(1,\infty)}:=\{\, g(x,T) \mid x\in(1,\infty)\,\}\subseteq\R[T]_{\leq d}\]
	The coefficients of the elements of $g_{(1,\infty)}$ are bounded, since they are images of polynomial mappings in $X^{-1}$ which is bounded on $(1,\infty)$. Therefore $g_{(1,\infty)}$ is a bounded set in $\R[T]_{\leq d}$. By the scalar case we therefore know that $B_{(1,\infty)}$ is also bounded, say entrywise. But that implies that non of the entries of $B$ can have positive degree in $X$. This proves the second inclusion in (1).

	\medskip
	Now take $F\subseteq\R[X,T]_{(k,d)}$ any bounded set, say coefficientwise bounded in the supremum norm on $[0,1]$, i.e.
	\[F_{[0,1]}:=\{\, f(x,T) \mid f\in F,\ x\in[0,1]\,\}\subseteq\R[T]_{\leq d}\]
	is bounded. For $M:=\ch_d^{-1}(F)\subseteq\R[X,T]_{(k,d)}$ we have
	\[M_{[0,1]}:=\{\, A(x) \mid A\in M,\ x\in[0,1]\,\}\subseteq \ch^{-1}(F_{[0,1]})\]
	which is also (entrywise) bounded, again by the scalar case. Therefore $M$ must be bounded, again say entrywise in the supremum norm on $[0,1]$. This proves (2).

	\medskip
	The closedness of the image of $\ch_{k,d}$ follows immediately from the properness, due to the fact that every finite dimensional $\R$-vector space is a locally compact Hausdorff space and therefore every proper map into it is closed, which is an easy exercise.
\end{proof}

From the smooth version and the previous lemmas we immediately get our main theorem:

\begin{thm}\label{thm:realrooted}
	$H_{k,d}=\im \ch_{k,d}$. In other words if $f\in\R[X,T]_{(k,d)}$ is $T$-real rooted, then there exists a symmetric $T$-spectral determinantal representations whose entries have at most degree $k$ in $X$.
\end{thm}
\begin{proof}
	Theorem~\ref{thm:smoothrealrooted} says that $H^+_{k,d}\subseteq \im\ch_d$. Applying (1) of Lemma~\ref{lem:proper} gives the degree bound, which tells us that already $H^+_{k,d}\subseteq\im\ch_{k,d}$. By (3) of the same lemma $\im\ch_{k,d}$ is closed and therefore contains $H_{k,d}$ by the density argument from Lemma~\ref{lem:approx}. The other inclusion, i.e. $\im\ch_{k,d}\subseteq H_{k,d}$ is the easy one and follows from the fact that real symmetric matrices have real eigenvalues only, and again Lemma~\ref{lem:proper} (1).
\end{proof}

The Helton-Vinnikov Theorem now follows directly.
\medskip

\begin{proof}[of Theorem~\ref{thm:heltonvinnikov}]
	Let $F\in\R[X,Y,Z]_d$ be hyperbolic with respect to $e\in\R^3$. By a linear change of variables and rescaling, we can assume that $e=(0,0,1)$ and $F$ is monic in $Z$. Now the hyperbolicity implies that $f:=F(X,1,T)$ is real rooted in $T$. Since $F$ is of total degree $d$ we have that $f\in\R[X,T]_{(1,d)}$, so $f\in H_{1,d}$. Therefore $f\in\im\ch_{1,d}$ by Theorem~\ref{thm:realrooted}, i.e. there exists a linear symmetric matrix $A=XB+C\in S_{1,d}$ such that $f$ is the characteristic polynomial of $A$. But then
	\[F=\det(Z-XB-YC)\]
	is a symmetric linear determinantal representation of $F$ that is definite at $e$.
\end{proof}

Note that for proving a smooth version of Helton-Vinnikov for strictly hyperbolic polynomials we only need the degree correspondence of Lemma~\ref{lem:proper} and not the rather technical approximation step. Moreover the approximation of hyperbolic polynomials by strictly hyperbolic polynomials defining smooth projective curves is somewhat simpler, because the concept of ``simple roots at infinity'' is a bit more straight forward as explained in Remark \ref{rem:realrooted_vs_hyperbolic} in Section~\ref{sec:approx}.

\bibliographystyle{plain}
\bibliography{lit}
\end{document}